\documentclass[12pt]{amsart}

\usepackage{amsmath,amssymb,amsfonts,fullpage,amssymb,amsthm}

\usepackage{verbatim}
\usepackage[usenames]{color}
\usepackage{hyperref}
\usepackage{url}

\newtheorem{thm}{Theorem}[section]
\newtheorem{prop}[thm]{Proposition} 
\newtheorem{lem}[thm]{Lemma}
\newtheorem{cor}[thm]{Corollary} 

\newtheorem*{claim}{Claim}

\theoremstyle{remark}
\newtheorem{rem}{Remark}[section]

\theoremstyle{definition}
\newtheorem{defn}[thm]{Definition}

\newtheorem*{notation}{Notation}

\newtheorem{coro}{Corollary}[section]
\theoremstyle{definition}
\newtheorem*{nota}{Notation}

\theoremstyle{remark}

\DeclareMathOperator{\depth}{depth}

\newcommand{\noprint}[1]{\relax}

\title{Local Ramsey theory. An abstract approach}
\author{Carlos Di Prisco, Jos\'e G. Mijares and Jes\'us Nieto}

\begin{document}

\maketitle

\begin{abstract}
Given a topological Ramsey space $(\mathcal R,\leq, r)$, we extend the notion of semiselective coideal to sets $\mathcal H\subseteq\mathcal R$ and study conditions for $\mathcal H$ that will enable us to make the structure $(\mathcal R,\mathcal H,\leq, r)$ a Ramsey space (not necessarily topological) and also study forcing notions related to $\mathcal H$ which will satisfy abstract versions of interesting properties of the corres-\linebreak ponding forcing notions in the realm of Ellentuck's space (see \cite{ellen,mathias}). This extends results from \cite{farah,Mijares} to the most general context of topological Ramsey spaces. As applications, we prove that for every topological Ramsey space $\mathcal R$, under suitable large cardinal hypotheses every semiselective ultrafilter $\mathcal U\subseteq\mathcal R$ is generic over $L(\mathbb R)$; and that given a semiselective coideal $\mathcal H\subseteq\mathcal R$, every definable subset of $\mathcal R$ is $\mathcal H$--Ramsey. This generalizes the corresponding results for the case when $\mathcal R$ is equal to Ellentuck's space (see \cite{farah,DMU}).
\end{abstract}

\section{Introduction}

Let $A\subseteq\mathbb{N}$ be given and consider the set $A^{[\infty]} =
\{X\subset A: |X| = \infty\}$. Consider the sets of the form $[a,A] = \{B\in\mathbb{N}^{[\infty]} : a\sqsubset B\subseteq A\},$ where $a$ is a finite $\mathbb{N}^{[\infty]}$,
$A\in\mathbb{N}^{[\infty]}$ and $a\sqsubset B$ means that $a$ is
an initial segment of $B$. For a family $\mathcal{H}\subseteq\mathbb{N}^{[\infty]}$, a set
$\mathcal{X}\subseteq \mathbb{N}^{[\infty]}$ is said to be
$\mathcal{H}$-{\it Ramsey} \ if for every nonempty $[a,A]$ with
$A\in\mathcal{H}$ there exists $B\in\mathcal{H}\cap A^{[\infty]}$
such that $[a,B]\subseteq\mathcal{X}$ or $[a,B]\cap\mathcal{X} =
\emptyset$. $\mathcal{X}$ is said to be $\mathcal{H}$-{\it Ramsey
null} \ if for every nonempty $[a,A]$ with $A\in\mathcal{H}$ there
exists $B\in\mathcal{H}\cap [a,A]$ such that
$[a,B]\cap\mathcal{X} = \emptyset$.

Local Ramsey theory includes the study and characterization of the
property defined above, which is a relativized version of the
completely Ramsey property (see \cite{galpri,ellen}). In \cite{mathias},
Mathias introduces the happy families (or selective
coideals) of subsets of $\mathbb{N}$ and relativizes the notion of
completely Ramsey subsets of $\mathbb{N}^{[\infty]}$ to such families. Then he proves that
analytic sets are $\mathcal{U}$-Ramsey when $\mathcal{U}$ is a
Ramsey ultrafilter and generalizes this result for arbitrary happy
families. In \cite{farah}, Farah gives an answer to the question
of Todorcevic: what are the combinatorial properties of the family
$\mathcal{H}$ of ground model subsets of $\mathbb{N}$ which
warranties diagonalization of the Borel partitions? This is done
by imposing a condition on $\mathcal{H}$ which is weaker than
selectivity: the notion of semiselectivity. Farah shows not only that the semiselectivity of $\mathcal{H}$ is enough to make the $\mathcal{H}$-Ramsey property equivalent to the abstract Baire property with respect
to $\mathcal{H}$, but also shows that this latter equivalence characterizes semiselectivity.  In \cite{Mijares}, a step toward the understanding of the local Ramsey property within the most general context of topological Ramsey spaces is done.

In this work, given a topological Ramsey space $(\mathcal R,\leq, r)$, we extend the notion of semiselective coideal to sets $\mathcal H\subseteq\mathcal R$ and study conditions for $\mathcal H$ that will enable us to make the structure $(\mathcal R,\mathcal H,\leq, r)$ a Ramsey space (not necessarily topological) and also study forcing notions related to $\mathcal H$ which will satisfy abstract versions of interesting properties of the corresponding forcing notions in the realm of Ellentuck's space, extending results from \cite{farah,Mijares} to the most general context of topological Ramsey spaces. As applications, we prove that for every topological Ramsey space $\mathcal R$, under suitable large cardinal hypotheses every semiselective ultrafilter $\mathcal U\subseteq\mathcal R$ is generic over $L(\mathbb R)$; and that given a semiselective coideal $\mathcal H\subseteq\mathcal R$, every definable subset of $\mathcal R$ is $\mathcal H$--Ramsey. This generalizes the corresponding results for the case when $\mathcal R$ is equal to Ellentuck's space (see \cite{farah,DMU}).

The structure of this work is as follows: Section
\ref{def ramsey spaces} is a short introduction to the theory of topological Ramsey spaces. In Section \ref{semisele}, we extend the notion of semiselective coideal to subsets $\mathcal H$ of a topological Ramsey space $\mathcal R$ and study conditions for $\mathcal H$ that will enable us to make the structure $(\mathcal R,\mathcal H,\leq, r)$ a Ramsey space. In particular, in this section we study the characterization of the corresponding abstract version of the $\mathcal H$--Ramsey property. In section \ref{souslinop}, it is shown that the
family of $\mathcal H$--Ramsey subsets of $\mathcal R$ is closed under
the Souslin operation, if $\mathcal H$ is semiselective. In Section \ref{sele} we introduce an abstract version of selective coideal. This is then connected with Section \ref{semisele forcing} where we study forcing notions related to semiselectivity as defined in Section \ref{semisele}. Finally, in Section \ref{app} we apply the main results in the previous Sections to prove that under suitable large cardinal hypotheses every semiselective ultrafilter $\mathcal U\subseteq\mathcal R$ is generic over $L(\mathbb R)$ and that given a semiselective coideal $\mathcal H\subseteq\mathcal R$, every definable subset of $\mathcal R$ is $\mathcal H$--Ramsey. 

\section{Topological Ramsey spaces}\label{def ramsey spaces}
The definitions and results throughout this section are
taken from \cite{Todo}.  A previous presentation can be found in \cite{carsimp}. 

\subsection{Metrically closed spaces and approximations}\label{ MC spaces} 
Consider a triplet of the form $(\mathcal{R}, \leq,
r )$, where $\mathcal{R}$ is a set, $\leq$
is a quasi order on $\mathcal{R}$ and $r: \mathbb{N}\times\mathcal{R}\rightarrow \mathcal{AR}$
 is a function with range $\mathcal{AR}$. For every $n\in \mathbb{N}$ and every $A\in \mathcal{R}$, 
 let us write 
\begin{equation}\label{eq axiomatic approximations}
      r_n(A) := r(n,A)
\end{equation}
We say that $r_{n}(A)$ is {\bf the} $n$th {\bf approximation of} $A$. 
We will reserve capital letters $A,B\dots$ for elements in $\mathcal{R}$ while lowercase letters $a,b\dots$
will denote elements of $\mathcal{AR}$.  In order to capture the combinatorial structure
 required to ensure the provability of an Ellentuck type Theorem, some assumptions on
 $(\mathcal{R}, \leq, r)$ will be imposed. The first is the following:\\[2mm]
\noindent \textbf{(A.1)}\ [{\bf Metrization}] 
\begin{itemize}
	\item[{(A.1.1)}]For any $A\in \mathcal{R}$, $r_{0}(A) = \emptyset$.
	\item[{(A.1.2)}]For any $A,B\in \mathcal{R}$, if $A\neq B$ then
	$(\exists n)\ (r_{n}(A)\neq r_{n}(B))$.
	\item[{(A.1.3)}]If $r_{n}(A) =r_{m}(B)$ then $n = m$ and $(\forall i<n)\ (r_{i}(A) = r_{i}(B))$.
\end{itemize}
Take the discrete topology on $\mathcal{AR}$ and endow 
$\mathcal{AR}^{\mathbb{N}}$ with the product topology; this is the metric space of all the sequences of 
elements of $\mathcal{AR}$. The set $\mathcal{R}$ can be identified with the corresponding image 
in $\mathcal{AR}^{\mathbb{N}}$. We will say that $\mathcal{R}$ is
{\bf metrically closed} if, as a subspace $\mathcal{AR}^{\mathbb{N}}$ with the
inherited topology, it is closed. The basic open sets generating the metric topology on $\mathcal{R}$ inherited from the product topology of $\mathcal{AR}^{\mathbb{N}}$ are of the form:
\begin{equation}\label{eq basic opens in ARN}
		[a] = \{B\in \mathcal{R} : (\exists n)(a =
	r_{n}(B))\}
\end{equation}
where $a\in\mathcal{AR}$.  Let us define the {\bf length} of $a$,
as the unique integer $|a|=n$ such that $a = r_{n}(A)$ for some $A\in
\mathcal{R}$. For every $n\in\mathbb N$, let 
\begin{equation}
    \mathcal{AR}_n : = \{a\in\mathcal{AR} : |a|=n\} 
\end{equation}
Hence, 
\begin{equation}
    \mathcal{AR} = \bigcup_{n\in\mathbb N} \mathcal{AR}_n
\end{equation}

The {\bf Ellentuck type neighborhoods} are of the form:
\begin{equation}\label{eq basic Ellentuck opens}
		[a,A] = \{B\in [a] : B\leq A\}=
		\{B\in \mathcal {R} : (\exists n)\ a=r_n(B) \ \&\ B\leq A\}
\end{equation}
where $a\in\mathcal{AR}$
and $A\in \mathcal{R}$. 

\medskip

We will use the symbol $[n,A]$ to abbreviate $[r_{n}(A),A]$. 

\medskip

Let 
\begin{equation}
    \mathcal{AR}\!\!\upharpoonright \!\! A = \{a\in\mathcal{AR} : [a,A]\neq\emptyset\}
\end{equation}
Given a neighborhood $[a,A]$ and $n\geq |a|$, let $r_n[a,A]$ be the image of $[a,A]$ by the function $r_n$, i.e., 
\begin{equation}\label{eq image of basic Ellentuck opens}
		r_n[a,A] =
		\{r_n(B) : B\in [a,A] \}
\end{equation}

\subsection{Ramsey sets}\label{ssection Ramsey sets}
A set $\mathcal{X}\subseteq \mathcal{R}$
is \textbf{Ramsey} if for every neighborhood $[a,A]\neq\emptyset$
there exists  $B\in [a,A]$ such that $[a,B]\subseteq \mathcal{X}$
or $[a,B]\cap \mathcal{X} = \emptyset$. A set
$\mathcal{X}\subseteq \mathcal{R}$ is \textbf{Ramsey null} if for
every neighborhood $[a,A]$ there exists  $B\in [a,A]$ such that
$[a,B]\cap \mathcal{X} = \emptyset$.

\subsection{Topological Ramsey spaces}\label{ssection Ramsey spaces}
We say that $(\mathcal{R}, \leq,r)$ is a \textbf{topological\linebreak Ramsey space} iff subsets of
$\mathcal{R}$ with the Baire property are Ramsey and\linebreak meager subsets of $\mathcal{R}$ are Ramsey null.

\medskip

Given $a,b\in\mathcal{AR}$, write 

\begin{equation}
a\sqsubseteq b\ \mbox{ iff } (\exists A\in\mathcal R)\ (\exists m,n\in\mathbb N)\  m\leq n, a=r_m(A)\mbox{ and } b= r_n(A).
\end{equation}

\medskip

By A.1, $\sqsubseteq$ can be proven to be a partial order on $\mathcal{AR}$.

\medskip

\noindent \textbf{(A.2)}\ [{\bf Finitization}] There is a quasi order $\leq_{fin}$ on
$\mathcal{AR}$ such that:
\begin{itemize}
    \item[(A.2.1)] $A\leq B$ iff
    $(\forall n)\ (\exists m) \ \ (r_{n}(A)\leq_{fin} r_{m}(B))$.
    \item[(A.2.2)] $\{b\in \mathcal{AR} : b\leq_{fin} a\}$ is finite, for every
    $a\in \mathcal{AR}$.
\item[(A.2.3)] If $a\leq_{fin} b$ and $c \sqsubseteq a$ then there is $d \sqsubseteq b$ such that $c \leq_{fin} d$.
\end{itemize}

\medskip

Given $A\in\mathcal{R}$ and $a\in\mathcal{AR}\!\!\upharpoonright \!\! A$,
we define the {\bf depth} of $a$ in $A$ as 
\begin{equation}\label{eq depth of a segment}
		\operatorname{depth}_{A}(a): = \min\{n :a\leq_{fin}r_n(A)\}
\end{equation}

\noindent \textbf{(A.3)}\ [{\bf Amalgamation}] Given $a$ and $A$
with $\operatorname{depth}_{A}(a)=n$, the following holds:
\begin{itemize}
		\item[(A.3.1)] $(\forall B\in [n,A])\ \ ([a,B]\neq\emptyset)$.
		\item[(A.3.2)] $(\forall B\in [a,A])\ \ (\exists A'\in [n,A])\ \
		([a,A']\subseteq [a,B])$.
\end{itemize}

\bigskip

\noindent \textbf{(A.4)}\ [{\bf Pigeonhole Principle}] Given $a$
and $A$ with $\operatorname{depth}_{A}(a) = n$, for every $\mathcal{O}\subseteq\mathcal{AR}_{|a|+1}$ there is $B\in
[n,A]$ such that $r_{|a|+1}[a,B]\subseteq\mathcal{O}$ or $r_{|a|+1}[a,B]\subseteq\mathcal{O}^c$.

\begin{thm}[Todorcevic, \cite{Todo}]\label{AbsEll}
	{\rm [{\bf Abstract Ellentuck Theorem}]}
	Any $(\mathcal{R}, \leq, r)$ with $\mathcal{R}$ metrically closed and satisfying (A.1)-(A.4) is a
	topological Ramsey space.
\end{thm}

Besides \cite{Todo}, we refer the reader to \cite{DT1,DT2,DMT,Mijares, Mijares2,Mijares3,MijaNie,MijaPad,Tim2} for further developments on the theory of Ramsey spaces.
\section{Abstract semiselectivity}\label{semisele}

\begin{notation}
Given a triple $(\mathcal R, \leq, r)$ as defined in the previous section and $\mathcal H\subseteq\mathcal R$, let $\mathcal H\!\!\upharpoonright \!\!A = \{B\in\mathcal H : B\leq A\}$.
\end{notation}

\begin{defn}
Consider a triple $(\mathcal R, \leq, r)$ satisfying ${\bf A1 - A4}$. Given $\mathcal H\subseteq\mathcal R$, we say that $\mathcal H$ is a \textbf{coideal} if it satisfies the following:
\begin{itemize}
\item[{(a)}] For all $A,B\in\mathcal R$, if $A\in\mathcal H$ and $A\leq B$ then $B\in\mathcal H$.
\item[{(b)}] (${\bf A3}\mod\mathcal H$) For all $A\in\mathcal H$ and $a\in\mathcal{AR}\!\!\upharpoonright \!\! A$, the following holds:
\begin{itemize}
\item $[a,B]\neq\emptyset$ for all $B\in [\depth_A(a),A]\cap\mathcal H$.
\item If $B\in\mathcal H\!\!\upharpoonright \!\!A$ and $[a,B]\neq\emptyset$ then there exists $A'\in [\depth_A(a), A]\cap\mathcal H$ such that  $\emptyset\neq [a,A']\subseteq [a,B]$.

\end{itemize}
\item[{(c)}]  (${\bf A4}\mod\mathcal H$) Let $A\in\mathcal H$ and $a\in\mathcal{AR}\!\!\upharpoonright \!\! A$ be given. For all $\mathcal O \subseteq \mathcal{AR}_{|a|+1}$ there exists $B\in [\depth_A(a),A]\cap\mathcal H$ such that $r_{|a|+1}[a,B]\subseteq\mathcal O$ or $r_{|a|+1}[a,B]\cap\mathcal O=\emptyset$.
\end{itemize}

\end{defn}

We will study conditions for a coideal $\mathcal H\subseteq \mathcal R$ such that the structure $(\mathcal R,\mathcal H,\leq, \leq, r, r)$ is a Ramsey space (in the sense of \cite{Todo}, Chapter 4). For short, from now on we will write\linebreak $(\mathcal R,\mathcal H,\leq, r)$ instead of $(\mathcal R,\mathcal H,\leq, \leq, r, r)$. It is easy to see that $(\mathcal R,\mathcal H,\leq, r)$ satisfies ${\bf A1 - A4}$ for general Ramsey spaces. Therefore, we know from the Abstract Ramsey Theorem (Theorem 4.27 in \cite{Todo}) that if $\mathcal H$ is closed in $(\mathcal{AR})^{\mathbb N}$ then $(\mathcal R,\mathcal H,\leq, r)$ is a Ramsey space. However, when $\mathcal H$ is not necessarily closed, we want to study conditions for $\mathcal H$ that will enable us to still make the structure $(\mathcal R,\mathcal H,\leq, r)$ a Ramsey space and will also allow us to study forcing notions related to $\mathcal H$ which will satisfy abstract versions of interesting properties of the corresponding forcing notions in the realm of Ellentuck's space.  One such condition is given in Definition  \ref{semiselective} below.

\medskip

From now on suppose that for a fixed $(\mathcal R, \leq, r)$, ${\bf A1 - A4}$ hold and
$\mathcal R$ is metrically closed. Hence $(\mathcal R, \leq, r)$ is a
topological Ramsey space. Let $\mathcal H\subseteq\mathcal R$ be a coideal. The natural definitions of $\mathcal H$-Ramsey and $\mathcal H$-Baire sets are:

\begin{defn} $\mathcal X\subseteq \mathcal R$ is \textbf{$\mathcal H$-Ramsey} if for every $[a,A]\neq \emptyset$, with $A\in \mathcal H$, there exists $B\in [a,A]\cap \mathcal H$ such that $[a,B]\subseteq \mathcal X$ or $[a,B]\subseteq \mathcal X^c$. If for every $[a,A]\neq \emptyset$, there exists $B\in [a,A]\cap\mathcal H$ such that $[a,B]\subseteq \mathcal X^c$; we say that  $\mathcal X$ is \textbf{$\mathcal H$-Ramsey null}.
\end{defn}

\begin{defn} $\mathcal X\subseteq \mathcal R$ is $\mathcal H$-\textbf{Baire} if for every $[a,A]\neq \emptyset$, with $A\in \mathcal H$, there exists $\emptyset\neq[b,B]\subseteq[a,A]$, with $B\in\mathcal H$, such that $[b,B]\subseteq \mathcal X$ or $[b,B]\subseteq \mathcal X^c$. If for every $[a,A]\neq \emptyset$, with $A\in\mathcal H$, there exists $\emptyset\neq[b,B]\subseteq[a,A]$, with $B\in \mathcal H$, such that $[b,B]\subseteq \mathcal X^c$; we say that $\mathcal X$ is $\mathcal H$-\textbf{meager}.
\end{defn}

It is clear that if $\mathcal X\subseteq \mathcal R$ is $\mathcal H$-Ramsey then $\mathcal X$ is $\mathcal H$-Baire.  

\begin{defn} Let $\mathcal H\subseteq\mathcal R$ be  a coideal. Given sets $\mathcal D, \mathcal S\subseteq\mathcal  H$, we say that  $\mathcal D$ is \textbf{dense open in} $\mathcal S$ if the following hold:
\begin{enumerate}
\item $(\forall A\in \mathcal S)\ (\exists B\in\mathcal D)\ B\leq A$.
\item  $(\forall A\in \mathcal S)\ (\forall B\in\mathcal D)\ [A\leq B \rightarrow A\in\mathcal D]$.
\end{enumerate}
\end{defn}

\begin{defn}\label{diagonal1}
Given $A\in\mathcal R$ and a family $\mathcal A=\{A_a\}_{a\in \mathcal{AR}\upharpoonright A}\subseteq\mathcal R$, we say that $B\in\mathcal R$ is a \textbf{diagonalization} of $\mathcal A$ if for every $a\in\mathcal{AR}\upharpoonright B$ we have $[a, B]\subseteq [a,A_a]$. 
\end{defn}

\begin{defn}\label{diagonal2}
 Given $A\in\mathcal H$ and a collection $\mathcal D = \{\mathcal D_a\}_{a\in \mathcal{AR}\upharpoonright A}$ such that each $\mathcal D_a$ is dense open in $\mathcal H\cap [\depth_A(a),A]$, we say that $B\in\mathcal R$ is a \textbf{diagonalization} of $\mathcal D$ if there exists a family $\mathcal A=\{A_a\}_{a\in \mathcal{AR}\upharpoonright A}$, with $A_a\in\mathcal D_a$, such that $B$ is a diagonalization of $\mathcal A$.
\end{defn}

\begin{defn}\label{semiselective}
We say that a coideal $\mathcal H\subseteq\mathcal R$ is \textbf{semiselective} if for every $A\in \mathcal H$, every collection $\mathcal D = \{\mathcal D_a\}_{a\in \mathcal{AR}\upharpoonright A}$ such that each $\mathcal D_a$ is dense open in $\mathcal H\cap [\depth_A(a),A]$ and every $B\in\mathcal H\!\!\upharpoonright \!\!A$, there exists $C\in\mathcal{H}$ such that $C$ is a diagonalization of $\mathcal D$ and $C\leq B$.
\end{defn}

\vspace{.25cm}

Our goal in this section is to prove that the families of $\mathcal H$-Ramsey sets and $\mathcal H$-Baire sets coincide, if $\mathcal H$ is semiselective (see Theorem \ref{baire-ramsey} below).  The following \textbf{combinatorial forcing} will be used: Fix $\mathcal F\subseteq \mathcal{AR}$. We say that $A\in \mathcal H$
\emph{\textbf{accepts}} $a\in \mathcal{AR}$ if for every $B\in \mathcal H\cap [\depth_A(a),A]$ there exists $n\in \mathbb{N}$
such that $r_n(B)\in\mathcal F$. We say that $A$
\emph{\textbf{rejects}} $a$ if for all $B\in [\depth_A(a),A]\cap\mathcal H$, $B$ does not accept $a$; and we say that $A$
\emph{\textbf{decides}} $a$ if $A$ either accepts or rejects $a$.

\begin{lem}\label{comb forcing}
The combinatorial forcing has the following properties:
\begin{enumerate}
\item If $A$ both accepts and rejects $a$ then $[a,A]=\emptyset$.
\item\label{accept heir} If $A$ accepts $a$ then every $B\in \mathcal H\!\!\upharpoonright \!\!A$ with $[a,B]\neq\emptyset$ accepts $a$.
\item\label{reject heir} If $A$ rejects $a$, then every $B\in \mathcal H\!\!\upharpoonright \!\!A$ with $[a,B]\neq\emptyset$ rejects $a$.
\item\label{some decide} For every $A\in \mathcal H$ and every $a\in \mathcal{AR}\!\!\upharpoonright \!\!A$ there exists $B\in [a,A]\cap\mathcal H$ which decides $a$.
\item\label{accepts longer} If $A$ accepts $a$ then $A$ accepts every $b\in r_{|a|+1}([a,A])$.
\item\label{not accept longer} If $A$ rejects $a$ then there exists $B\in [a,A]\cap\mathcal H$ such that $A$ does not accept any $b\in r_{|a|+1}([a,B])$.
\end{enumerate}
\end{lem}
\begin{proof}
(1)--(5) follows from the definitions. Let us prove (\ref{not accept longer}):

\medskip

Suppose that $A$ rejects $a$. Let $\mathcal O=\{b\in\mathcal{AR} : A \mbox{ accepts } b\}$. By {\bf A.4} mod $\mathcal H$, there exists $B\in\mathcal H\cap [a,A]$ such that $r_{|a|+1}[a,B]\subseteq \mathcal O$ or $r_{|a|+1}[a,B]\subseteq \mathcal O^c$. If the first alternative holds then take $C\in\mathcal H\cap [a,B]$. Let $b=r_{|a|+1}(C)$. Then $b\in\mathcal O$ and therefore $A$ accepts $b$. Since $C\in[b,A]$ then there exists $n$ such that $r_n(C)\in\mathcal F$. Therefore $B$ accepts $a$, because $C$ is arbitrary. But this contradicts that $A$ rejects $a$. Hence, $r_{|a|+1}[a,B]\subseteq \mathcal O^c$ and we are done.
\end{proof}

\begin{claim}
 Given $A\in\mathcal H$, there exists $D\in \mathcal H \!\!\upharpoonright \!\!A$ which decides every $b\in \mathcal{AR} \!\!\upharpoonright \!\!D$.
\end{claim}

\begin{proof} For every $a\in \mathcal{AR} \!\!\upharpoonright \!\!A$ define
$$\mathcal D_a=\{C\in\mathcal H\cap [\depth_A(a),A]\colon C {\mbox{ decides }}a\}$$
By parts \ref{accept heir}, \ref{reject heir} and \ref{some decide} of Lemma \ref{comb forcing} each $\mathcal D_a$ is dense open in $\mathcal H\cap [\depth_A(a),A]$. By semi-\linebreak selectivity, there exists $D\in \mathcal H\!\!\upharpoonright \!\!A$ which diagonalizes the collection $(\mathcal D_a)_{a\in \mathcal{AR}\upharpoonright A}$. By parts \ref{accept heir} and \ref{reject heir} of Lemma \ref{comb forcing}, $D$ decides every $a\in\mathcal{AR} \!\!\upharpoonright \!\!D$.
\end{proof}

The following is an abstract version of the semisective Galvin lemma (see \cite{galvin,farah}).

\begin{lem}[Semiselective abstract Galvin's lemma]\label{galvinlocal}
Given $\mathcal F\subseteq\mathcal{AR}$, a semiselective coideal $\mathcal H\subseteq \mathcal{R}$, and $A\in\mathcal H$, there exists $B\in \mathcal H\!\!\upharpoonright \!\!A$ such that one of the following holds:
\begin{enumerate}
\item $\mathcal{AR} \!\!\upharpoonright \!\!B\cap\mathcal F=\emptyset$, or
\item $\forall C\in [\emptyset,B]$ $(\exists \ n\in \mathbb{N})$ $(r_n(C)\in\mathcal F)$.
\end{enumerate}
\end{lem}

\begin{proof} Consider $D$ as in the Claim. If $D$ accepts $\emptyset$ part (2) holds and we are done. So assume that $D$ rejects $\emptyset$ and for $a\in \mathcal{AR} \!\!\upharpoonright \!\!D$ define
$$\mathcal D_a=\{C\in\mathcal H\cap [\depth_A(a),D]\colon C {\mbox{ rejects every }}b\in r_{|a|+1}([a,C])\}$$if $D$ rejects $a$, and $\mathcal D_a=\mathcal H\cap [\depth_A(a),D]$, otherwise. By parts \ref{reject heir} and \ref{not accept longer} of Lemma \ref{comb forcing} each $\mathcal D_a$ is dense open in $\mathcal H\cap [\depth_A(a),D]$. By semiselectivity, choose $B\in \mathcal H\!\!\upharpoonright \!\!D$ such that for all $a\in\mathcal{AR}\!\!\upharpoonright \!\!B$ there exists $C_a\in\mathcal D_a$ with $[a,B]\subseteq [a,C_a]$. For every $a\in\mathcal{AR}\!\!\upharpoonright \!\!B$, $C_a$ rejects all $b\in r_{|a|+1}([a,C_a])$. So $B$ rejects all $b\in r_{|a|+1}([a,B])$: given one such $b$, choose any $\hat{B}\in\mathcal H\cap [b,B]$. Then $\hat{B}\in\mathcal H\cap [b,C_a]$. Therefore, since $C_a$ rejects $b$, $\hat{B}$ does not accept $b$.

Hence, $B$ satisfies that $\mathcal{AR}\!\!\upharpoonright \!\!B\cap\mathcal F=\emptyset$. This completes the proof of the Lemma.
\end{proof}

\begin{notation}
$\mathcal{AR}\!\!\upharpoonright \!\![a,B] = \{b\in\mathcal{AR} : a\sqsubseteq b\ \&\ (\exists n\geq |a|)(\exists C\in [a,B])\ b=r_n(C)\}$.

\end{notation}

In a similar way we can prove the following generalization of lemma \ref{galvinlocal}: 

\begin{lem}\label{galvinlocal2}
Given a semiselective coideal $\mathcal H$ of $\mathcal{R}$, $\mathcal F\subseteq\mathcal{AR}$,  $A\in\mathcal H$ and $a\in\mathcal{AR}\!\!\upharpoonright \!\!A$, there exists $B\in \mathcal H\cap [a,A]$ such that one of the following holds:
\begin{enumerate}
\item $\mathcal{AR}\!\!\upharpoonright \!\![a,B]\cap\mathcal F=\emptyset$, or
\item $\forall C\in [a,B]$ $(\exists \ n\in \mathbb{N})$ $(r_n(C)\in\mathcal F)$.
\end{enumerate}
\end{lem}

Now, we give an application of Lemma \ref{galvinlocal}. Recall from Equation \ref{eq basic opens in ARN} that the basic metric open subsets of $\mathcal R$ are of the form $[b]=\{A\in\mathcal R\colon b\sqsubset A\}$, where $b\sqsubset A$ means\linebreak  $(\exists n\in\mathbb{N})\ (r_n(A)=b)$.

\begin{thm}\label{abiertos}
Suppose that $\mathcal{H}\subseteq\mathcal{R}$ is a semiselective coideal. Then the metric open subsets of $\mathcal R$ are $\mathcal H$-Ramsey.
\end{thm}

\begin{proof}
Let $\mathcal{X}$ be a metric open subset of $\mathcal{R}$ and fix a nonempty $[a,A]$ with $A\in\mathcal{H}$. Without a loss of generality, we can assume $a = \emptyset$. Since $\mathcal{X}$ is open, there exists $\mathcal{F}\subseteq \mathcal{AR}$ such that $\mathcal X = \bigcup_{b\in \mathcal F} [b]$. Let $B\in \mathcal H\!\!\upharpoonright \!\!A$ be as in Lemma \ref{galvinlocal}. If (1) holds then $[0,B]\subseteq\mathcal{X}^c$ and if (2) holds then $[0,B]\subseteq\mathcal{X}$.
\end{proof}

The following is one of the main results of this work. 

\begin{thm}\label{baire-ramsey}
If $\mathcal{H}\subseteq\mathcal{R}$ is a semiselective coideal then $\mathcal X\subseteq\mathcal R$ is $\mathcal H$--Ramsey iff $\mathcal X$ is $\mathcal H$--Baire
\end{thm}

\begin{proof} Let $\mathcal X$ be a $\mathcal H$--Baire subset of $\mathcal R$ and consider $A\in \mathcal H$. As before, we only proof the result for $[\emptyset,A]$ without a loss of generality. For $a\in\mathcal{AR} \!\!\upharpoonright \!\!A$ define
$$\mathcal D_a=\{B\in[\depth_A(a),A]\cap\mathcal H\colon [a,B]\subseteq \mathcal X {\mbox { or }} [a,B]\subseteq \mathcal X^c$$ $${\mbox { or }} [(\forall C\in[a,B])\  [a,C]\cap\mathcal X\neq\emptyset {\mbox { and }} [a,C]\cap\mathcal X^c\neq\emptyset]\}$$
It is easy to see that each $\mathcal D_a$ is dense open in $\mathcal H\cap [\depth_A(a),A]$. By semiselectivity, choose $B\in\mathcal H\!\!\upharpoonright \!\!A$ which diagonalizes the collection $(\mathcal D_a)_{a\in\mathcal{AR} \upharpoonright A}$. Let $\mathcal F_0=\{a\in\mathcal{AR} \!\!\upharpoonright \!\!A\colon [a,B]\subseteq \mathcal X\}$ and $\mathcal F_1=\{a\in\mathcal{AR} \!\!\upharpoonright \!\!A\colon [a,B]\subseteq \mathcal X^c\}$. Consider $B_0\in\mathcal H\!\!\upharpoonright \!\!B$ as in Lemma \ref{galvinlocal} applied to $\mathcal F_0$ and $B$. If (2) of Lemma \ref{galvinlocal} holds then $[\emptyset,B_0]\subseteq\mathcal X$ and we are done. So assume that (1) holds. That is, $\mathcal{AR}\!\!\upharpoonright \!\!B_0\cap \mathcal F_0 = \emptyset$. Now consider $B_1$ as in Lemma \ref{galvinlocal} applied to $\mathcal F_1$ and $B_0$. Again, if (2) holds then $[\emptyset,B_1]\subseteq\mathcal X^c$ and we are done. Notice that $\mathcal{AR}(B_1)\cap \mathcal F_1\neq\emptyset$ because $\mathcal{AR}\!\!\upharpoonright \!\!B_1\cap \mathcal F_0 = \emptyset$ and $\mathcal X$ is $\mathcal H$--Baire. So (2) holds. This concludes the proof.
\end{proof}


\begin{rem}
In vitue of Theorem \ref{baire-ramsey}, if $\mathcal H$ is a semiselective coideal then $(\mathcal R,\mathcal H, \leq, r)$ is a Ramsey space. It should be clear that the axioms {\bf A1 - A4} for general Ramsey spaces are satisfied (see Section 4.2 in \cite{Todo}) from the definition of coideal and from the fact that $(\mathcal R, \leq, r)$ satisfies axioms {\bf A1 - A4}  for topological Ramsey spaces. But we are using semiselectivity (and the fact that $\mathcal R$ is closed) instead of asking that $\mathcal H$ be closed. We will get more insight into the forcing notion $(\mathcal H,\leq^*)$ and other related forcing notions in this way. 
\end{rem}

The following is a local version of Theorem 1.6 from \cite{Mijares}, which is an abstract version of Ramsey's Theorem \cite{Ramsey}:

\begin{thm}\label{ramsey2}
Suppose that $\mathcal H\subseteq \mathcal R$ is a semiselective coideal. Then, given a partition $f\colon \mathcal{AR}_2\to \{0,1\}$ and $A\in \mathcal H$, there exists $B\in \mathcal H\!\!\upharpoonright \!\!A$ such that $f$ is constant on $\mathcal{AR}_2(B)$.
\end{thm}

\begin{proof} Let $f$ be the partition $\mathcal{AR}_2=\mathcal C_0\cup \mathcal C_1$, and consider $A\in \mathcal H$. Define
$$\mathcal D_a=\{B\in [\depth_A(a),A]\cap\mathcal H\colon f {\mbox { is constant on }}r_2[a,B]\}$$
if $a\in \mathcal{AR}_1\!\!\upharpoonright \!\!A$ and $\mathcal D_a=\mathcal H\cap[\depth_A(a),A]$, otherwise. Using {\bf A.4} mod $\mathcal{H}$ in the case $a\in \mathcal{AR}_1\!\!\upharpoonright \!\!A$, it is easy to prove that each $\mathcal{D}_a$ is dense open in $\mathcal H\cap[\depth_A(a),A]$. By semiselectivity,  there exists $B_1\in \mathcal H\!\!\upharpoonright \!\!A$ which diagonalizes  the collection $(\mathcal D_a)_{a\in\mathcal{AR} \upharpoonright A}$. Notice that for every $a\in \mathcal{AR}_1\!\!\upharpoonright \!\!B_1$, there exists $i_a\in \{0,1\}$ such that $r_2[a,B_1]\subseteq \mathcal C_{i_a}$. Now, consider the partition $g\colon \mathcal{AR}_1\to \{0,1\}$ defined by $g(a)=i_a$ if $a\in \mathcal{AR}_1\!\!\upharpoonright \!\!B_1$. By {\bf A.4} mod $\mathcal H$ there exists $B\in \mathcal H\cap [0,B_1]$ such that $g$ is constant on $r_1[0,B]=\mathcal{AR}_1(B)$. But $B\leq B_1\leq A$, so $B$ is as required.
\end{proof}

\begin{defn}\label{ramsey coideal}
A coideal $\mathcal H\subseteq \mathcal R$ is Ramsey if  for every integer $n\geq 2$, every partition $f\colon \mathcal{AR}_n\to \{0,1\}$ and $A\in \mathcal H$, there exists $B\in \mathcal H\!\!\upharpoonright \!\!A$ such that $f$ is constant on $\mathcal{AR}_n(B)$.
\end{defn}

Proceeding in a similar way, using induction, we can prove the following generalization of Theorem \ref{ramsey2}:

\begin{thm}\label{ramsey3}
Every semiselective coideal $\mathcal H\subseteq \mathcal R$ is Ramsey.
\end{thm}


\section{The Souslin operation}\label{souslinop}

\begin{defn} \label{souslin}
The result of applying the \textbf{Souslin operation} to a family $(\mathcal{X}_a)_{a\in\mathcal{AR}}$ of subsets of $\mathcal{R}$ is:
$$\bigcup_{A\in\mathcal{R}}\bigcap_{n\in\mathbb{N}}\mathcal{X}_{r_n(A)}$$
\end{defn} 

The goal of this section is to show that the family of $\mathcal H$--Ramsey subsets of $\mathcal R$ is closed under the Souslin operation when $\mathcal H$ is a semiselective coideal. 

\begin{lem}\label{sigmaideal}
If $\mathcal H\subseteq\mathcal R$ is a semiselective coideal of $\mathcal{R}$ then the families of $\mathcal H$--Ramsey and $\mathcal H$--Ramsey null subsets of $\mathcal R$ are closed under countable unions.
\end{lem}

\begin{proof} Fix $[a,A]$ with $A\in\mathcal H$. We will suppose that $a=\emptyset$ without a loss of generality. Let $(\mathcal X_n)_{n\in\mathbb{N}}$ be a sequence of $\mathcal H$--Ramsey null subsets of $\mathcal R$. Define for $a\in\mathcal{AR} \!\!\upharpoonright \!\!A$
$$\mathcal D_a=\{B\in\mathcal H\cap [a,A]\colon [a,B]\subseteq \mathcal X_n^c \ \ \forall n \leq |a| \}$$
Every $\mathcal D_a$ is dense open in $\mathcal H\cap [a,A]$, so let $B\in\mathcal H\!\!\upharpoonright \!\!A$ be a diagonalization of $(\mathcal D_a)_a$. Then $[\emptyset, B]\subseteq \bigcap_n \mathcal X_n^c$. Thus, $\bigcup_n \mathcal X_n$ is $\mathcal H$--Ramsey null. Now, suppose that $(\mathcal X_n)_{n\in\mathbb{N}}$ is a sequence of $\mathcal H$--Ramsey subsets of $\mathcal R$. If there exists $B\in\mathcal H\!\!\upharpoonright \!\!A$ such that $[\emptyset,B]\subseteq \mathcal X_n$ for some $n$, we are done. Otherwise, using an argument similar to the one above, we prove that $\bigcup_n \mathcal X_n$ is $\mathcal H$--Ramsey null.
\end{proof}Let

\begin{equation}
Exp(\mathcal H)=\{[n,A]\colon n\in \mathbb{N}, A\in \mathcal H\}.
\end{equation}

\begin{defn}
We say $\mathcal{X}\subseteq \mathcal{R}$ is $Exp(\mathcal{H})$--\textbf{nowhere dense} if every member of $Exp(\mathcal{H})$ has a subset in $Exp(\mathcal{H})$ that is disjoint from $\mathcal{X}$.
\end{defn}

Notice that every $\mathcal{H}$--Ramsey null set is $Exp(\mathcal{H})$--nowhere dense. And every $Exp(\mathcal{H})$--nowhere dense is $\mathcal{H}$--meager. Thus, if $\mathcal{H}$ is semiselective, every $Exp(\mathcal{H})$--nowhere dense is $\mathcal{H}$--Ramsey. But if $\mathcal{X}$ is both $Exp(\mathcal{H})$--nowhere dense and $\mathcal{H}$--Ramsey, it has to be $\mathcal{H}$--Ramsey null. As a consequence of Lemma \ref{sigmaideal} and Theorem \ref{baire-ramsey} we have:

\begin{coro}
If $\mathcal H\subseteq\mathcal R$ is a semiselective coideal of $\mathcal{R}$ and $\mathcal{X}\subseteq \mathcal{R}$ then the following are equivalent:
\begin{enumerate}
\item $\mathcal{X}$ is $\mathcal{H}$--Ramsey null
\item $\mathcal{X}$ is $Exp(\mathcal{H})$--nowhere dense.
\item $\mathcal{X}$ is $Exp(\mathcal{H})$--meager (i. e. countable union of $Exp(\mathcal{H})$--nowhere dense sets).
\item $\mathcal{X}$ is $\mathcal{H}$--meager.
\end{enumerate}
\end{coro}

\qed

Given a set $X$, say that two subsets $A,B$ of $X$ are ``\emph{compatible}'' with respect to a family $\mathcal F$ of subsets of $X$ if there exists $C\in \mathcal F$ such that $C\subseteq A\cap B$. And $\mathcal F$ is \emph{M-like} if for $\mathcal G\subseteq \mathcal F$ with $|\mathcal G|<|\mathcal F|$, every member of $\mathcal F$ which is not compatible with any member of $\mathcal G$ is compatible with $X\setminus \bigcup \mathcal G$. A $\sigma$-algebra $\mathcal A$ of subsets of $X$ together with a $\sigma$-ideal $\mathcal A_0\subseteq\mathcal A$ is a \emph{Marczewski pair} if for every $A\subseteq X$ there exists $\Phi(A)\in \mathcal A$ such that $A\subseteq \Phi(A)$ and for every $B\subseteq \Phi(A)\setminus A$, $B\in\mathcal A\Rightarrow B\in \mathcal A_0$. The following is a well known fact:

\begin{thm}[Marczewski]\label{marcz}
Every $\sigma$-algebra of sets which together with a $\sigma$-ideal is a Marczeswki pair, is closed under the Souslin operation.
\end{thm}
\qed

Let $\mathcal H$ be a semiselective coideal of $\mathcal{R}$. Then we have:

\begin{prop}\label{mlike}
The family $Exp(\mathcal H)$ is $M$-like.
\end{prop}

\begin{proof} Consider $\mathcal B\subseteq Exp(\mathcal H)$ with $|\mathcal B|<|Exp(\mathcal H)|=2^{\aleph_0}$ and suppose that $[a,A]$ is not compatible with any member of $\mathcal B$, i. e. for every $[b,B]\in \mathcal B$, $[b,B]\cap [a,A]$ does not contain any member of $Exp(\mathcal H)$. We claim that $[a,A]$ is compatible with $\mathcal R\smallsetminus \bigcup \mathcal B$. In fact:

Since $|\mathcal B|<2^{\aleph_0}$, by Lemmas \ref{galvinlocal2} and {sigmaideal}, $\bigcup \mathcal B$ is $\mathcal H$-Ramsey and therefore $\mathcal H$-Baire, by Theorem \ref{baire-ramsey}. So, there exist $[b,B]\subseteq [a,A]$ with $B\in \mathcal H$ such that:
\begin{enumerate}
\item $[b,B]\subseteq \bigcup \mathcal B$ or
\item $[b,B]\subseteq \mathcal R\smallsetminus \bigcup \mathcal B$
\end{enumerate}

(1) is not possible because $[a,A]$ is not compatible with any member of $\mathcal B$. And (2) says that $[a,A]$ is compatible with $\mathcal R\smallsetminus \bigcup \mathcal B$. This completes the proof.
\end{proof}

\medskip

Proposition \ref{mlike} says that the family of $\mathcal H$--Ramsey subsets of $\mathcal R$ together with the family of $\mathcal H$--Ramsey null subsets of $\mathcal R$ is a Marczewski pair. Thus, by theorem \ref{marcz}, we obtain the following:

\begin{thm}\label{souslin}
The family of $\mathcal H$--Ramsey subsets of $\mathcal R$ is closed under the Souslin operation.
\end{thm}

\qed

\section{Abstract Selectivity}\label{sele}

\begin{defn}Given $A\in \mathcal{H}$ and $\mathcal{A}=(A_a)_{a\in\mathcal{AR}\upharpoonright A}\subseteq \mathcal{H}\!\!\upharpoonright\!\! A$ with $[a,A_a]\neq\emptyset$ for all $a$, we say that $\mathcal{A}$ is \textbf{filtered by} $\leq$ if for every $a$, $b\in\mathcal{AR}\!\!\upharpoonright \!\!A$ there exists $c\in\mathcal{AR}\!\!\upharpoonright \!\!A$ such that $A_c\leq A_a$ and $A_c\leq A_b$.  
\end{defn}

\begin{defn} A coideal $\mathcal H\subseteq\mathcal R$ is \textbf{selective} if given $A\in \mathcal H$, for every $\mathcal{A}=(A_a)_{a\in\mathcal{AR}\upharpoonright A}\subseteq\mathcal{H}\upharpoonright \!\!A$ filtered by $\leq$ such that $[a,A_a]\neq\emptyset$ for all $a$, there exists $B\in\mathcal H\!\!\upharpoonright \!\!A$ which diagonalizes $\mathcal{A}$.
\end{defn}

\begin{lem}\label{filtered}
Given a coideal $\mathcal{H}$ of $\mathcal{R}$ and $A\in\mathcal{H}$, for every $(\mathcal{D}_a)_{a\in\mathcal{AR}\upharpoonright A}$ such that each $\mathcal{D}_a$ is dense open in $\mathcal H\cap[a,A]$ there exists $(A_a)_{a\in\mathcal{AR}\upharpoonright A)}$ filtered by $\leq$ such that $A_a\in\mathcal{D}_a$ for all $a\in\mathcal{AR}\!\!\upharpoonright \!\!A$.  
\end{lem}

\begin{proof}
For every $k\in\mathbb{N}$, list 
$$\mathcal{A}_k=\{a_1^1,a_2^1,\dots,a_{n_k}^1\}=\{a\in\mathcal{AR}\!\!\upharpoonright \!\!A\colon depth_A(a)=k\}$$
(every $\mathcal{A}_k$ is finite by {\bf A2}). Since each $\mathcal{D}_a$ is dense open in $\mathcal H\cap[a,A]$, using {\bf A3} mod $\mathcal{H}$) we can choose $A^{1,1}\in\mathcal{D}_{a_1^1}\!\!\upharpoonright\!\!A$, $A^{1,2}\in\mathcal{D}_{a_2^1}\!\!\upharpoonright\!\!A^{1,1}$, $\dots$, $A^{1,n_1}\in\mathcal{D}_{a_{n_1}^1}\!\!\upharpoonright\!\!A^{1,n_1-1}$. Again,  we can choose $A^{2,1}\in\mathcal{D}_{a_1^2}\!\!\upharpoonright\!\!A^{1,n_1}$, $A^{2,2}\in\mathcal{D}_{a_2^2}\!\!\upharpoonright\!\!A^{2,1}$, $\dots$, $A^{2,n_2}\in\mathcal{D}_{a_{n_2}^2}\!\!\upharpoonright\!\!A^{2,n_2-1}$. And so on. Then $(A^{i,j})_{ij}$ is as required.
\end{proof}

\begin{prop} If $\mathcal H\subseteq\mathcal R$ is a selective coideal then $\mathcal H$ is semiselective.
\end{prop}

\begin{proof} Consider $A\in\mathcal H$ and let $\mathcal{D}=(\mathcal D_a)_{a\in \mathcal{AR}\upharpoonright A}$ be such that each $\mathcal{D}_a$ is dense open in $\mathcal H\cap[a,A]$. It is clear that if $\hat{B}\in\mathcal H\!\!\upharpoonright \!\!A$ then  $\mathcal{D}_a$ is dense open in $\mathcal H\cap[a,\hat{B}]$, for all $a\in\hat{B}$. Using lemma \ref{filtered} we can build $\mathcal{A}=(A_a)_{a\in\mathcal{AR}\upharpoonright\hat{B})}$ filtered by $\leq$ such that $A_a\in\mathcal{D}_a$ for every $a\in\mathcal{AR}\upharpoonright\hat{B}$. By selectivity, there exists $B\in\mathcal H\!\!\upharpoonright \!\!\hat{B}$ which diagonalizes $\mathcal{A}$.
\end{proof}

\section{Abstract semiselectivity and forcing}\label{semisele forcing}

\subsection{Forcing with $(\mathcal H,\leq^*)$}

\begin{nota}
We will borrow the the following notation from Section 2 of \cite{Mijares}. For $A,B\in\mathcal R$, write $A\leq^*B$ if there exists $a\in\mathcal{AR}\!\!\upharpoonright \!\! A$ such that $[a,A]\subseteq [a,B]$. In this case we say that $A$ is an \textit{almost-reduction} of $B$. This is a generalization of \textit{almost-inclusion} and \textit{almost-condensation} (see \cite{blass}). In \cite{Mijares}, it is proved that $(\mathcal R,\leq^*)$ is reflexive and transitve.
\end{nota}

In this section we will describe the main properties of the forcing notion $(\mathcal H,\leq^*)$, for a semiselective $\mathcal H$. To do so, we will need to consider a special type of coideal $\mathcal U\subseteq\mathcal R$:

\begin{defn}\label{ultra}
 Given $\mathcal U\subseteq\mathcal R$, we say that $\mathcal U$ is an \textbf{ultrafilter} if it satisfies the following:
\begin{itemize}
\item[{(a)}] $\mathcal U$ is a \textit{filter} on $(\mathcal R,\leq)$. That is:

\begin{enumerate}
\item For all $A,B\in\mathcal R$, if $A\in\mathcal U$ and $A\leq B$ then $B\in\mathcal U$.
\item\label{ultra finite intersection} For all $A,B\in\mathcal U$ and $a\in\mathcal{AR}$, if $[a,A]\neq\emptyset$ and $[a,B]\neq\emptyset$, then there exists $C\in\mathcal U$ such that $C\in [a,A]\cap[a,B]$. In particular, for all $A,B\in\mathcal U$, there exists $C\in\mathcal U$ such that $C\leq A$ and $C\leq B$.
\end{enumerate}

\item[{(b)}] If $\mathcal U'\subseteq\mathcal R$ is a filter on $(\mathcal R,\leq)$ and $\mathcal U\subseteq \mathcal U'$ then $\mathcal U'=\mathcal U$. That is, $\mathcal U$ is a \textit{maximal filter} on $(\mathcal R,\leq)$.

\item[{(c)}]  (${\bf A3}\mod\mathcal U$) For all $A\in\mathcal U$ and $a\in\mathcal{AR}\!\!\upharpoonright \!\! A$, the following holds:
\begin{itemize}
\item $[a,B]\neq\emptyset$ for all $B\in [\depth_A(a),A]\cap\mathcal U$.
\item If $B\in\mathcal U\!\!\upharpoonright \!\!A$ and $[a,B]\neq\emptyset$ then there exists $A'\in [\depth_A(a),A]\cap\mathcal U$ such that  $\emptyset\neq [a,A']\subseteq [a,B]$.

\end{itemize}

\end{itemize}

\end{defn}

\begin{rem}
In \cite{Tim2}, a similar abstract definition of ultrafilter is used. But the ultrafilters used in \cite{Tim2} do not satisfy part (c) of Definition \ref{ultra}. 
\end{rem}

It turns out that every such ultrafilter $\mathcal U\subseteq\mathcal R$ also satisfies the following very useful condition and therefore it is a coideal.

\begin{itemize}
\item[{(d)}]  (${\bf A4}\mod\mathcal U$) Let $A\in\mathcal U$ and $a\in\mathcal{AR}\!\!\upharpoonright \!\! A$ be given. For all $\mathcal O \subseteq \mathcal{AR}_{|a|+1}$ then there exists $B\in [\depth_A(a),A]\cap\mathcal U$ such that $r_{|a|+1}[a,B]\subseteq\mathcal O$ or $r_{|a|+1}[a,B]\cap\mathcal O=\emptyset$.
\end{itemize}






\begin{lem}\label{sigma d}
If $\mathcal H\subseteq\mathcal R$ is a semiselctive coideal then $(\mathcal H,\leq^*)$ is $\sigma$-distributive.
\end{lem}
\begin{proof}
For every $n\in\mathbb N$, let $\mathcal D_n\subseteq\mathcal H$ be dense open in $(\mathcal H,\leq^*)$. Fix $A\in\mathcal H$. For all $a\in\mathcal{AR}\!\!\upharpoonright \!\! A$, the set $\mathcal D_a=\{B\in\mathcal H\cap[\depth_A(a),A] : B\in\mathcal D_{|a|}\}$ is dense open in $\mathcal H\cap[\depth_A(a),A]$: Fix $a\in\mathcal{AR}\!\!\upharpoonright \!\! A$. Obviously, if $B\in \mathcal D_a$ and $B'\in\mathcal H\cap[\depth_A(a),A]$ is such that $B'\leq B$ then $B'\in \mathcal D_a$. On the other hand, given $C\in\mathcal H\cap[\depth_A(a),A]$, choose $B_a\in\mathcal D_{|a|}$ such that $B_a\leq^*C$. Then there exists $b\in \mathcal{AR}\!\!\upharpoonright \!\!B_a$ such that $[b,B_a]\subseteq[b,C]$. We will assume that $\depth_C(b)\geq\depth_C(a)=\depth_A(a)$ (otherwise, let $m=|b|+\depth_C(a)$ and choose $D\in [b,B_a]$. Let $\hat{b}=r_{m+1}(D)$. Then $[\hat{b},B_a]\subseteq[\hat{b},C]$ and $\depth_C(\hat{b})\geq\depth_C(a)$). By ${\bf A3}\mod\mathcal H$, choose $B\in\mathcal H\cap[\depth_C(b),C]$ such that $\emptyset\neq[b,B]\subseteq[b,B_a]$. So $B\leq^*B_a$ and therefore $B\in\mathcal D_{|a|}$. Notice also that since $\depth_C(b)\geq\depth_C(a)=\depth_A(a)$ then $B\in\mathcal H\cap[\depth_A(a),C]\subseteq\mathcal H\cap[\depth_A(a),A]$. This implies that $B\leq C$ and $B\in\mathcal D_a$. This completes the proof that $\mathcal D_a$ is dense open. Let $B\in\mathcal H\!\!\upharpoonright \!\! A$ be a diagonalization of $\{\mathcal D_a \}_{a\in\mathcal{AR}\upharpoonright A}$. Then there exists a family $\{A_a\}_{a\in\mathcal{AR}\upharpoonright A}$ with $A_a\in\mathcal D_a$ such that $[a,B]\subseteq[a,A_a]$ for all $a\in\mathcal{AR}\!\!\upharpoonright \!\! B$. This means that $B\leq^* A_a$ for all $a\in\mathcal{AR}\!\!\upharpoonright \!\! B$. Therefore, $B\in\mathcal D_{|a|}$ for all $a\in\mathcal{AR}\!\!\upharpoonright \!\! B$. That is, $B\in\bigcap_n\mathcal D_n$. This completes the proof.

\end{proof}

\begin{lem}\label{ramsey u}
Let $\mathcal R$ be a topological Ramsey space and $\mathcal H\subseteq R$ be a semiselective coideal. Forcing with $(\mathcal H,\leq^* )$ adds no new elements of $\mathcal{AR}^{\mathbb N}$ (in particular, no new elements of $\mathcal R$ or $\mathcal H$), and if\,  $\mathcal U$ is the  $(\mathcal H,\leq^* )$--generic filter over some
ground model $V$, then $\mathcal U$ is a Ramsey ultrafilter in $V[\mathcal U]$.
\end{lem}
\begin{proof}
Since is $(\mathcal H,\leq^* )$ $\sigma$-distributive, the fact that forcing with $(\mathcal H,\leq^* )$ adds no new elements of $\mathcal{AR}^{\mathbb N}$ follows by a standard argument. See for instance \cite{jech}, Theorem 2.10. Let $\mathcal U$ be the  $(\mathcal H,\leq^* )$--generic filter over some ground model $V$. By genericity, $\mathcal U$ is a maximal filter. Also by genericity, ${\bf A3}$ (for the space $\mathcal R$) and Theorem \ref{ramsey3}, we have that ${\bf A3}$ mod $\mathcal U$ holds (and therefore, $\mathcal U$ satisfies Definition \ref{ultra}) and $\mathcal U$ is Ramsey.
\end{proof}

\begin{lem}\label{selective u}
Let $\mathcal U$ be the  $(\mathcal H,\leq^* )$--generic filter over some
ground model $V$. Then $\mathcal U$ is selective in $V[\mathcal U]$.
\end{lem}

\begin{proof}
In $V[\mathcal U]$, fix $A\in\mathcal U$ and let $\{A_a\}_{a\in\mathcal{AR}\upharpoonright A}$ be a collection of elements of $\mathcal U$ such that $[a,A_a]\neq\emptyset$, for all $a\in\mathcal{AR}$. Given an integer $n>0$, define $f : \mathcal{AR}_{n+1} \rightarrow\{0,1\}$ as $f(b)=1$ if and only if $[b,A_{r_n(b)}]\neq\emptyset$.  By Lemma \ref{ramsey u}, $\mathcal U$ is a Ramsey ultrafilter in $V[\mathcal U]$ so there exist $B\in\mathcal U$ such that $B\leq A$ and $f$ is constant in $\mathcal{AR}_{n+1}\!\!\upharpoonright \!\! B$. Take an arbitrary $a\in \mathcal{AR}_n\!\!\upharpoonright \!\! B$. Notice that there exists $C\in\mathcal U$ such that $C \leq B$, $C\leq A_a$, and $[a,C]\neq\emptyset$, by part (a)(\ref{ultra finite intersection}) of Definition \ref{ultra}. For any $b \in r_{n+1}[a,C]$ we have $f(b) = 1$, and therefore $f$ takes constant value $1$ in $\mathcal{AR}_{n+1}\!\!\upharpoonright \!\! B$.  Since $a$ is arbitrary, this implies that for all $a\in\mathcal{AR}_n\!\!\upharpoonright \!\! B$,  $r_{n+1}[a,B]\subseteq  r_{n+1}[a,A_a]$. 

\medskip

Now, recall that $\mathcal U\subseteq\mathcal H$. So the above reasoning implies that for every integer $n>0$, the set $$\mathcal E_n = \{B\in\mathcal H : (\forall a\in\mathcal{AR}_n\!\!\upharpoonright \!\! B)\ r_{n+1}[a,B]\subseteq  r_{n+1}[a,A_a]\}$$ is dense open in $(\mathcal H,\leq^*)$. So, by genericity of $\mathcal U$ and $\sigma$--distributivity of $(\mathcal H,\leq^*)$, we can choose $B\in\mathcal U\cap\bigcap_n\mathcal E_n $. So for every $a\in\mathcal{AR}\!\!\upharpoonright \!\! B$ we have $r_{|a|+1}[a,B]\subseteq  r_{|a|+1}[a,A_a]$. By ${\bf A1}$, this implies that $[a,B]\subseteq  [a,A_a]$, for every  $a\in\mathcal{AR}\!\!\upharpoonright \!\! B$. Therefore, $B$ is a diagonalization of $(A_a)_{a\in\mathcal{AR}}$. This completes the proof.
\end{proof}

\begin{lem}
Suppose $\mathcal H$ is not semiselective. Let $\mathcal U$ be $(\mathcal H,\leq^* )$--generic filter over some
ground model $V$. Then $\mathcal U$ is not selective in $V[\mathcal U]$.
\end{lem}
\begin{proof}
Since $\mathcal H$ is not semiselective, there exist $A\in\mathcal H$ and a collection $(\mathcal D_a)_{a\in\mathcal{AR}\upharpoonright A}$ such that $\mathcal D_a$ is dense in $\mathcal H\cap[\depth_A(a),A]$, for every  $a\in\mathcal{AR}\!\!\upharpoonright \!\!A$, with no diagonalization in $\mathcal H$.  In $V[\mathcal U]$, it turns out that each $\mathcal D_a$ is dense in $(\mathcal H\!\!\upharpoonright \!\!A,\leq^* )$. Proceeding as in Lemma \ref{filtered}, find a collection $(A_a)_{a\in\mathcal{AR}\upharpoonright A}$ filtered by $\leq^*$ and such that $A_a\in\mathcal U\cap\mathcal D_a$, for every  $a\in\mathcal{AR}\!\!\upharpoonright \!\!A$. Then the collection $(A_a)_{a\in\mathcal{AR}\upharpoonright A}$ has no diagonalization in $\mathcal U$ and therefore $\mathcal U$ is not selective. This completes the proof.
\end{proof}

The following theorem summarizes all the results of this section.

\begin{thm}
Let $\mathcal H\subseteq\mathcal R$ be a coideal. The following are equivalent.
\begin{enumerate}
\item $\mathcal H$ is semiselective.
\item Forcing with $(\mathcal H,\leq^* )$ adds no new elements of $\mathcal{AR}^{\mathbb N}$ (in particular, no new elements of $\mathcal R$ or $\mathcal H$), and if $\mathcal U$ is the  $(\mathcal H,\leq^* )$--generic filter over some
ground model $V$, then $\mathcal U$ is a selective ultrafilter in $V[\mathcal U]$.
\end{enumerate}
\end{thm}

\subsection{Forcing with $\mathbb M_{\mathcal H}$}\label{M_U}

Let $\mathcal H$ be a semiselective coideal. In this section we will study the forcing notion induced by the following poset:  $$\mathbb M_{\mathcal H}=\{(a,A) : A\in\mathcal H,\ a\in\mathcal{AR} \!\!\upharpoonright \!\!A \}\cup\{\emptyset\},$$ where $(a,A)\leq (b,B)$ if and only if $[a,A]\subseteq [b,B]$. We say that $\mathbb M_{\mathcal H}$ is the \textbf{Mathias poset} associated to $\mathcal H$. It is said that $\mathcal H$ has the \textbf{Prikry property} if for every sentence of the forcing language $\phi$ and every condition $(a,A)\in\mathbb M_{\mathcal H}$ there exists $B\in [a,A]\cap\mathcal H$ such that $(a,B)$ decides $\phi$. We say that  $x\in\mathcal R$ is $\mathbb M_{\mathcal H}$-\textbf{generic} over a model $V$ if for every dense open subset $\mathcal D\in M$ of $\mathbb M_{\mathcal H}$, there exists a condition $(a,A)\in\mathcal D$ such that $x\in[a,A]$. It is said that $\mathcal H$ has the  \textbf{Mathias property} if it satisfies that if $x$ is  $\mathbb M_{\mathcal H}$-generic over a model $V$,  then every $y\leq x$ is $\mathbb M_{\mathcal H}$-generic over $M$

\begin{thm}\label{prikry prop}
If $\mathcal H\subseteq\mathcal R$ is a semiselective coideal then it has the Prikry property. 
\end{thm}

\begin{proof}

Suppose $\mathcal H$ is semiselective and fix a sentence $\phi$ of the forcing language, and a condition $(a,A)\in\mathbb M_{\mathcal H}$. For every $b\in\mathcal{AR} \!\!\upharpoonright \!\!A$ with $b\sqsupseteq a$, let

$$\mathcal D_b = \{B\in\mathcal H\cap[\depth_A(b),A] : (b,B)\mbox{ decides } \phi\mbox{ or}\ (\forall C\in\mathcal H\cap [b,B])\ (b,C)\mbox{ does not decide } \phi\}.$$

\noindent and set $\mathcal D_b = \mathcal H\cap [\depth_A(b),A] $, for all $b\in\mathcal{AR} \!\!\upharpoonright \!\!A$ with $b\not\sqsupseteq a$. Each $\mathcal D_b$ is dense open in $\mathcal H\cap[\depth_A(b),A]$. Fix a diagonalization $B\in\mathcal H\!\!\upharpoonright \!\! A$. For every $b\in\mathcal{AR} \!\!\upharpoonright \!\!A$. Let 

$$\mathcal F_0 = \{b\in \mathcal{AR} \!\!\upharpoonright \!\!B : a\sqsubseteq b\ \ \&\ (b,B) \mbox{ forces }  \phi\},$$

$$\mathcal F_1 = \{b\in\mathcal{AR} \!\!\upharpoonright \!\!B : a\sqsubseteq b\ \ \&\ (b,B) \mbox{ forces }  \neg\phi\}.$$

Let $\hat{C}\in\mathcal H\!\!\upharpoonright \!\!B$ as in Lemma \ref{galvinlocal2} applied to $B$ and $\mathcal F_0$. And let $C\in\mathcal H\!\!\upharpoonright \!\!\hat{C}$ be as in Lemma \ref{galvinlocal} applied to $\hat{C}$ and $\mathcal F_1$. Let us prove that $(a,C)$ decides $\phi$. So let $(b_0,C_0)$ and $(b_1,C_1)$ be two different arbitrary extensions of $(a,C)$. Suppose that $(b_0,C_0)$ forces $\phi$ and $(b_1,C_1)$ forces $\neg\phi$. Then $b_0\in\mathcal F_0$ and $b_1\in\mathcal F_1$. But $b_0,b_1\in\mathcal{AR}\!\!\upharpoonright \!\!C$, so by the choice of $C$ this means that every element of $\mathcal H\cap[a,C]$ has an initial segment both in $\mathcal F_0$ and in $\mathcal F_1$. So there exist two compatible extensions of $(a,C)$ such that one forces $\phi$ and the other forces $\neg\phi$. A contradiction. So either both $(b_0,C_0)$ and $(b_1,C_1)$ force $\phi$ or both $(b_0,C_0)$ and $(b_1,C_1)$ force $\neg\phi$. Therefore $(a,C)$ decides $\phi$.

\end{proof}

Now, we will prove that if $\mathcal H\subseteq\mathcal R$ is semiselective then it has the Mathias property (see Theorem \ref{mathias prop} below). Given a selective ultrafilter $\mathcal U\subset\mathcal R$, let $\mathbb M_{\mathcal U}$ be set of all pairs $(a,A)$ such that $A\in\mathcal U$ and $[a,A]\neq\emptyset$. Order $\mathbb M_{\mathcal U}$ with the same ordering used for $\mathbb M_{\mathcal H}$. The results contained in the rest of this Section were essentially proved in \cite{Mijares}, but they were written (in \cite{Mijares}) for the case $\mathcal H=\mathcal R$ and using weaker definitions of ultrafilter and selectivity. We adapted the proofs and include them here for completeness.

\begin{defn}\label{captures}
 Let $\mathcal U\subseteq\mathcal R$ be a selective ultrafilter, $\mathcal D$ a dense open subset of $\mathbb M_{\mathcal U}$, and $a\in\mathcal{AR}$. We say that $A$ captures $(a,\mathcal D)$ if $A\in\mathcal U$, $[a,A]\neq\emptyset$, and for all $B\in [a,A]$ there exists $m>|a|$ such that $(r_m(B),A)\in\mathcal D$.
\end{defn}

\begin{lem}\label{lemma captures}
Let $\mathcal U\subseteq\mathcal R$ be a selective ultrafilter and $\mathcal D$ a dense open subset of $\mathbb M_{\mathcal U}$. Then, for every $a\in\mathcal{AR}$ there exists $A$ which captures $(a,\mathcal D)$.
\end{lem}
\begin{proof}

Given $a\in\mathcal{AR}$, choose $B\in\mathcal U$ such that $[a,B]\neq\emptyset$. We can define a collection $(C_b)_{b\in\mathcal{AR}\upharpoonright B}$, filtered by $\leq$ and with $[b,C_b]\neq\emptyset$, such that for all $b\in\mathcal{AR}\!\!\upharpoonright \!\! B$ with $a\sqsubseteq b$ either $(b,C_b)\in\mathcal D$ or $C_b=B$. By selectivity, let $C\in\mathcal U\cap[a,B]$ be a diagonalization of $(C_b)_{b\in\mathcal{AR}\upharpoonright B}$. Then, for all $b\in\mathcal{AR}\!\!\upharpoonright \!\! B$ with $a\sqsubseteq b$, if there exists a $\hat{C}\in\mathcal U$ such that $(b,\hat{C})\in\mathcal D$, we must have $(b,C)\in\mathcal D$.  Let $\mathcal X=\{D\in\mathcal R : D\leq C\rightarrow (\exists b\in\mathcal{AR}\!\!\upharpoonright \!\! D)\ a\sqsubset b\ \&\ (b,C)\in\mathcal D\}$. $\mathcal X$ is a metric open subset of $\mathcal R$ and therefore, by Theorem \ref{abiertos}, it is $\mathcal U$-Ramsey. Take $\hat{C}\in\mathcal U\cap[\depth_C(a),C]$ such that $[a,\hat{C}]\subseteq\mathcal X$ or $[a,\hat{C}]\cap\mathcal X=\emptyset$. We will show that the first alternative holds: Pick $A\in\mathcal U\cap[a,\hat{C}]$ and $(a',A')\in\mathcal D$ such that $(a',A')\leq (a,A)$. Notice that $a\sqsubseteq a'$ and therefore, as we pointed out at the end of the previous paragraph, we have $(a',C)\in\mathcal D$. By the definition of $\mathcal X$, we also have $A'\in\mathcal X$. Now choose $A''\in\mathcal U\cap[a,A']$. Then $(a', A'')$ is also in $\mathcal D$ and therefore $A''\in\mathcal X\cap[a,\hat{C}]$. This implies $[a,\hat{C}]\subseteq\mathcal X$. Finally, that $A$ captures $(a,\mathcal D)$ follows from the definition of $\mathcal X$ and the fact that $[a,A]\subseteq [a,\hat{C}]\subseteq [a,\hat{C}]$. This completes the proof.
\end{proof}

\begin{thm}\label{forcing M_U}
Let $\mathcal U\subseteq\mathcal R$ be a selective ultrafilter, in a given transitive model of $ZF+DCR$. Forcing over $M$ with $\mathbb M_{\mathcal U}$ adds a generic $g\in\mathcal R$ with the property that $g\leq^*A$ for all $A\in\mathcal U$. In fact, $B\in\mathcal R$ is $\mathbb M_{\mathcal U}$-generic over $V$ if and only if $B\leq^* A$ for all $A\in\mathcal U$. Also, $V[\mathcal U][g]=V[g]$.
\end{thm}
\begin{proof}

Suppose that $B\in\mathcal R$ is $\mathbb M_{\mathcal U}$-generic over $V$. Fix an arbitrary $A\in\mathcal U$. The set $\{(c,C)\in\mathbb M_{\mathcal U} : C\leq^*A\}$ is dense open and is in $V$: Fix $(a,A')\in\mathbb M_{\mathcal U}$. Choose $C_0\in\mathcal U$ such that $C_0\leq A, A'$. Since $\mathcal U$ is an ultrafilter, we can choose $C_1\mathcal U$ and $n\in \omega$ such that $[n,C_1]\subseteq[a,A']\cap[1,C_0]$. Let $c=r_n(C_1)$. By ${\bf A3}$ mod $\mathcal U$, there exists $C_2\in\mathcal U\cap[\depth_{A'}(c),A']$ such that $\emptyset\neq[c,C_2]\subseteq[c,C_1]$. It is clear that $[c,C_2]\subseteq[c,A]$ and therefore $C_2\leq^*A$. Also, since $\depth_{A'}(c)\geq\depth_{A'}(a)$, we have $[a,C_2]\neq\emptyset$. Thus,  $(a,C_2)\leq (a,A')$. That is, $\mathcal D$ is dense. It is obviously open. So, by genericity, there exists one $(c,C)\in\mathcal D$ such that $B\in[c,C]$. Hence $B\leq^*A$.

Now, suppose that $B\in\mathcal R$ is such that $B\leq^*A$ for all $A\in\mathcal U$, and let $\mathcal D$ be a dense open subset of $\mathbb M_{\mathcal U}$. We need to find $(a,A)\in\mathcal D$ such that $B\in[a,A]$. In $V$, by using Lemma \ref{lemma captures} iteratively, we can define a sequence $(A_n)_n$ such that $A_n\in\mathcal U$, $A_{n+1}\leq A_n$, and $A_n$ captures $(r_n(B),\mathcal D)$. Since $\mathcal U$ is in $V$ and selective, we can choose $A\in\mathcal U$, in $V$, such that $A\leq^*A_n$ for all $n$. By our assupmtion on $B$, we have $B\leq^*A$. So there exists an $a\in\mathcal{AR}$ such that $[a,B]\subseteq[a,A]$. Let $m=\depth_B(a)$. By ${\bf A3}$ mod $\mathcal U$, we can assume that $a=r_m(B)=r_m(A)$, and also that $A\in[r_m(B),A_m]$. Therefore, $B\in[m,A]$ and $A$ captures $(r_m(B),\mathcal D)$. Hence, the following is true in $V$: 

\begin{equation}\label{eq captures}
(\forall C\in[m,A]) (\exists n>m) ((r_n(C),A)\in\mathcal D).
\end{equation}

Let $\mathcal F=\{b : (\exists n>m)(b\in r_n[m,A]\ \&\ (b,A)\notin\mathcal D)\}$ and give $\mathcal F$ the strict end-extension ordering $\sqsubset$. Then the relation $(\mathcal F,\sqsubset)$ is in $V$, and by equation \ref{eq captures} $(\mathcal F,\sqsubset)$ is well-founded. Therefore, by a well-known argument due to Mostowski, equation \ref{eq captures} holds in the universe. Hence, since $B\in[m,A]$, there exists $n>m$ such that $(r_n(B),A)\in\mathcal D$. But $B\in[r_n(B),A]$, so $B$ is $\mathbb M_{\mathcal U}$--generic over $V$.

Finally, if $g$ is $\mathbb M_{\mathcal U}$--generic over $V$, then, in $V$, $\mathcal U=\{A\in V : A\in\mathcal R\ \&\ g\leq^*A\}$ and therefore $V[\mathcal U][g]=V[g]$.

\end{proof}
\begin{cor} \label{Mu generic}
If $B$ is $\mathbb M_{\mathcal U}$--generic over some model $V$ and $A\leq B$ then $A$ is also $\mathbb M_{\mathcal U}$--generic over $V$.
\end{cor}

\begin{lem}\label{Mu iteration}
Let $\mathcal H\subseteq\mathcal R$ be a semiselective coideal. Consider the forcing notion $\mathbb P = (\mathcal H,\leq^*)$ and let $\hat{\mathcal U}$ be a $\mathbb P$--name for a $\mathbb P$--generic ultrafilter. Then the iteration $\mathbb P *\mathbb M_{\hat{\mathcal U}}$ is equivalent to the forcing $\mathbb M_{\mathcal H}$.
\end{lem}
\begin{proof}

Recall that $\mathbb P *\mathbb M_{\hat{\mathcal U}}=\{(B,(\dot{a},\dot{A})) : B\in\mathcal H\ \&\ B\vdash (\dot{a},\dot{A})\in \mathbb M_{\hat{\mathcal U}}\}$, with the ordering  $(B,(\dot{a},\dot{A}))\leq (B_0,(\dot{a}_0,\dot{A}_0)) \Leftrightarrow B\leq^* B_0 \ \&\ (\dot{a},\dot{A})\leq (B_0,(\dot{a}_0,\dot{A}_0)$. The mapping $(a,A) \rightarrow (A, (\hat{a}, \hat{A}))$ is a dense embedding (see \cite{golds}) from $\mathbb M_{\mathcal H}$ to $\mathbb P *\mathbb M_{\hat{\mathcal U}}$ (here $\hat{a}$ and $\hat{A}$ are the canonical $\mathbb P$-names for $a$ and $A$, respectively): It is easy to show that this mapping preserves the order. So, given $(B,(\dot{a},\dot{A}))\in\mathbb P *\mathbb M_{\hat{\mathcal U}}$, we need to find $(d,D)\in\mathbb M_{\mathcal H}$ such that $(D, (\hat{d}, \hat{D}))\leq (B,(\dot{a},\dot{A}))$. Since $\mathbb P$ is $\sigma$-distributive, there exists $a\in\mathcal{AR}$, $A\in\mathcal H$ and $C\leq* B$ in $\mathcal H$ such that $C\vdash_{\mathbb P} (\hat{a}=\dot{a}\ \&\ \hat{A}=\dot{A})$ (so we can assume $a\in\mathcal{AR}\!\!\upharpoonright \!\! C$). Notice that $(C, (\hat{a}, \hat{A}))\in\mathbb P *\mathbb M_{\hat{\mathcal U}}$ and $(C, (\hat{a}, \hat{A})), (C, (\hat{a}, \hat{A}))\leq (B,(\dot{a},\dot{A}))$. So, $C\vdash_{\mathbb P} \hat{C}\in\hat{\mathcal U}$ and $C\vdash_{\mathbb P} \hat{A}\in\hat{\mathcal U}$. Then, $C\vdash_{\mathbb P} (\exists x\in\hat{\mathcal U})(x\in[\hat{a},\hat{A}]\ \&\ x\in[\hat{a},\hat{C}]$. So there exists $D\in\mathcal H$ such that $D\in[a,A]\cap[a,C]$. Hence, $(D, (\hat{a}, \hat{D}))\leq (B,(\dot{a},\dot{A}))$. This completes the proof.
\end{proof}

\bigskip

The next Theorem follows inmediately from  Corollary \ref{Mu generic}  and Lemma \ref{Mu iteration}.

\begin{thm}\label{mathias prop}
If $\mathcal H\subseteq\mathcal R$ is a semiselective coideal then it has Mathias property.
\end{thm}

\section{Some applications}\label{app}

\subsection{Selectivity and genericity}

Let $\mathcal R$ be a topological Ramsey space. In this Section we will show that if the existence of a super compact cardinal is consistent then so is the statement ``every semiselective ultrafilter $\mathcal U\subseteq\mathcal R$ is generic over $L(\mathbb R)$". First, let us state the following definition (see \cite{farah}, Definition 4.10; and also \cite{FMSh,QMW,ShW}).

\begin{defn}
Let $M$ be a countable elementary submodel of some structure of the form $H_{\theta}$ which contains a poset $\mathcal P$ and a $\mathcal P$--name $\hat{r}$ for a real. Then we say that $M$ is $(L(\mathbb R), \mathcal P)$--\textbf{correct} if for every $(M,\mathcal P)$--generic filter $G\subseteq \mathcal P\cap M$ and every formula $\phi(x,\overrightarrow{p})$ with parameter $\overrightarrow{p}$ in $M$, the formula $\phi(val_G(\hat{r}),\overrightarrow{p})$ is true in $L(\mathbb R)$ if and only if there exists a condition in $G$ which forces this. We say that \textbf{truth in  $L(\mathbb R)$ is unchangeable by forcing} if the following condition is satisfied: For every poset $\mathcal P$ there exists $\theta$ large enough so that there exist stationarily many countable ementary submodels $M$ of  $H_{\theta}$ which are $(L(\mathbb R), \mathcal P)$--correct.
\end{defn}

We will also need the following two lemmas. For the proof of Lemma \ref{truth} see \cite{FMSh,QMW,ShW}. And for the proof of Lemma \ref{truth reals} see \cite{farah}.

\begin{lem}\label{truth}
If there exist a supercompact cardinal, then truth in $L(\mathbb R)$ is unchangeable by forcing.
\end{lem}

\begin{lem}\label{truth reals}
Assume that truth in $L(\mathbb R)$ is unchangeable by forcing. If $E$ is a ccc space and $f: E\rightarrow \mathbb R$ is continuous then for every set of reals  $\mathcal X$ from $L(\mathbb R)$, the set $f^{-1}(\mathcal X)$ has the Baire property.
\end{lem}

Now we are ready to prove the following. 
\begin{thm}\label{complete comb}
If there exists a super compact cardinal, then every selective coideal $\mathcal U\subseteq\mathcal R$ is $(\mathcal R,\leq^*)$--generic over $L(\mathbb R)$.
\end{thm}
\begin{proof}
This is a generalization of Todorcevic's proof of the corresponding result for $\mathcal R = \mathbb N^{[\infty]}$ (see Theorem 4.9 of \cite{farah}). Let $\mathcal U\subseteq\mathcal R$ be a selective ultrafilter. Let $E$ be the topological space $\mathcal R$ with the topology generated by the family $Exp(\mathcal U)=\{[a,A] : a\in\mathcal{AR}, A\in\mathcal U\}$ (it is a topology because $\mathcal U$ is an ultrafilter). By Theorem \ref{baire-ramsey}, $E$ is a Baire space and the Baire subsets of $E$ are exactly the $\mathcal U$--Ramsey sets. $E$ is also a ccc space because the partial order $\mathbb M_{\mathcal U}$ is ccc.  The identity function $i: E\rightarrow \mathcal R$ is continuous, if we consider $\mathcal R$ with the metric separable topology inherited from $\mathcal{AR}^{\mathbb N}$. So by Lemma \ref{truth reals}, every set of reals $\mathcal X$ in $L(\mathbb R)$ is $\mathcal U$--Ramsey. In particular, if such $\mathcal X$ is dense in $(\mathcal R,\leq^*)$, then there exists $A\in\mathcal U$ such that $[\emptyset, A]\subseteq\mathcal X$. Therefore, $\mathcal U$ is $(\mathcal R,\leq^*)$--generic over $L(\mathbb R)$.
\end{proof}

\subsection{$\mathcal H$--Ramseyness of definable sets}

From Theorem \ref{complete comb} it is easy now to prove the following, which lifts Theorem 4.8 of \cite{farah} to the most general context. By definable sets we mean elements of $L(\mathbb R)$. The proof is very similar to that of Theorem 4.8 of \cite{farah} so we will leave the details to the reader.

\begin{thm}
If there exists a super compact cardinal and $\mathcal H\subseteq\mathcal R$ is a semiselective coideal, then all definable subsets of $\mathcal R$ are $\mathcal H$--Ramsey.
\end{thm}

Nevertheless, in \cite{DMU}, it was proved that the supercompact cardinal is not needed for the case $\mathcal R=\mathbb N^{[\infty]}$. Namely:

\begin{thm}[Di Prisco, Mijares, Uzc\'ategui; \cite{DMU}]\label{dmu}
Suppose $\lambda$ is a weakly compact cardinal. Let $V[G]$ be a generic extension by $Col(\omega,<\lambda)$ such that $\mathbb N^{[\infty]}\in V[G]$. Then, if $\mathcal H\subseteq\mathbb N^{[\infty]}$ is a semiselective coideal in $V[G]$, every subset of $\mathbb N^{[\infty]}$ in $L(\mathbb R)$ of $V[G]$  is $\mathcal H$-Ramsey. 
\end{thm}

Now we will use the results in Section \ref{M_U} to generalize Theorem \ref{dmu} to the context of any topological Ramsey space:

\begin{thm}
Let $\mathcal R$ be a topological Ramsey space. Suppose $\lambda$ is a weakly compact cardinal. Let $V[G]$ be a generic extension by $Col(\omega,<\lambda)$ such that $\mathcal R\in V[G]$. Then, if $\mathcal H\subseteq\mathcal R$ is a semiselective coideal in $V[G]$, every subset of $\mathcal R$ in $L(\mathbb R)$ of $V[G]$  is $\mathcal H$-Ramsey. 
\end{thm}

\begin{coro}
Let $\mathcal R$ be a topological Ramsey space. If there is a weakly compact cardinal then for every semiselective coideal $\mathcal H\subseteq\mathcal R$ all definable subsets of $\mathcal R$ are $\mathcal H$-Ramsey.
\end{coro}

\end{document}